\documentclass[reqno]{amsart}

\usepackage{enumitem}
\usepackage{hyperref}
\usepackage[all]{xy}
\usepackage{graphicx}
\usepackage{subfigure}
\usepackage{amsthm, amssymb, amsmath,mathrsfs}
\usepackage[utf8]{inputenc}

\newtheorem{theorem}{Theorem}[section]

\newtheorem{remark}{Remark}[section]

\newcommand{\R}{\mathbb{R}}
\newcommand{\N}{\mathbb{N}}
\newcommand{\Z}{\mathbb{Z}}

\newcommand{\nn}{\nonumber}
\newcommand{\be}{\begin{equation}}
\newcommand{\ee}{\end{equation}}
\newcommand{\ti}{\tilde}

\newcommand{\E}{\mathrm{e}}
\newcommand{\I}{\mathrm{i}}

\newcommand{\tl}{\mathrm{TL}}
\newcommand{\km}{\mathrm{KM}}

\newcommand*{\mailto}[1]{\href{mailto:#1}{\nolinkurl{#1}}}
\newcommand{\arxiv}[1]{\href{http://arxiv.org/abs/#1}{arXiv:#1}}

\newcommand{\msc}[1]{\href{http://www.ams.org/msc/msc2010.html?t=&s=#1}{#1}}

\numberwithin{equation}{section}

\begin{document}

\title[On uniqueness properties]{On uniqueness properties of solutions of the Toda and Kac--van Moerbeke hierarchies}

\author[I. Alvarez-Romero]{Isaac Alvarez-Romero}
\address{Department of Mathematical Sciences,
Norwegian University of
Science and Technology, NO--7491 Trondheim, Norway}
\email{\mailto{isaac.romero@math.ntnu.no}\\
\mailto{isaacalrom@gmail.com}}

\author[G. Teschl]{Gerald Teschl}
\address{Faculty of Mathematics\\ University of Vienna\\
Oskar-Morgenstern-Platz 1\\ 1090 Wien\\ Austria\\ and International 
Erwin Schr\"odinger Institute for Mathematical Physics\\ 
Boltzmanngasse 9\\ 1090 Wien\\ Austria}
\email{\mailto{Gerald.Teschl@univie.ac.at}}
\urladdr{\url{http://www.mat.univie.ac.at/~gerald/}}

\thanks{{\it Research supported by the Norwegian Research Council project DIMMA 213638.}}
\thanks{Discrete Contin. Dyn. Syst. {\bf 37}, 2259--2264 (2017)}


\begin{abstract}
We prove that  a solution of the Toda lattice cannot decay too fast at two different times
unless it is trivial. In fact, we establish this result for the entire Toda and Kac--van Moerbeke hierarchies.
\end{abstract}

\maketitle

\section{Introduction}

To set the stage recall the Toda lattice \cite{ta} (in Flaschka's variables \cite{fl1})
\begin{align} \nn
\dot{a}(n,t) &= a(n,t) \Big(b(n+1,t)-b(n,t)\Big), \\ \label{todeqfl}
\dot{b}(n,t) &= 2 \Big(a(n,t)^2-a(n-1,t)^2\Big), \qquad n\in \Z,
\end{align}
where the dot denotes a derivative with respect to $t$.
It is a well-studied physical model and one of the prototypical discrete integrable wave equations.
We refer to the monographs \cite{fad}, \cite{tjac}, \cite{ta} or the review articles \cite{krt}, \cite{taet}
for further information.

Existence and uniqueness for the Cauchy problem in the case of bounded initial data is well known and
so is the solution by virtue of the inverse scattering transform in the case of decaying initial data.
In particular, using the latter it can be show that (sufficiently fast) decaying initial conditions eventually
split into a finite number of solitons plus a decaying dispersive part (see \cite{krt}). Moreover, it is
also known that the Toda lattice preserves certain types of spatial asymptotic behavior \cite{ttd}.
However, this fact is restricted to polynomial or at most exponential type decay. On the other hand,
it is also known that compact support can occur for at most one time \cite{ttd} (see also \cite{KT}) and this clearly raises
the question if this assumption can be weakened to a certain decay instead. In fact, such results
are known for other nonlinear wave equations; see the introduction in \cite{ik} for the case of the nonlinear
Schr\"odinger equation and \cite{ekpv}, \cite{kpv2} for the generalized KdV equation.

It is the purpose of the present paper to fill this gap by establishing the following result:

\begin{theorem}
Let $a_0>0$, $b_0\in\R$ be given constants and
let $a(t), b(t)$ be a solution of the Toda lattice satisfying
\begin{equation}\label{eq:decay}
\sum_{n\in\mathbb{Z}}|n|\big(|a(n,t)-a_0|+|b(n,t)-b_0|\big)<\infty
\end{equation}
for one (and hence for all) $t\in\R$.
Suppose that for two different times $t_0<t_1$ and some constant $\delta>0$, we have
\begin{equation}
\sum_{n\geq M}\big(|a(n,t_j)-a_0|+|b(n,t_j)-b_0|\big)\leq C\frac{1}{M^{(1+\delta)2M}},\quad M>0, \quad j\in\{0,1\}.
\end{equation}
Then
\begin{equation*}
a(n,t)=a_0, \qquad b(n,t)=b_0, \quad (n,t)\in\Z\times\R.
\end{equation*}
\end{theorem}

\begin{remark}
(i). Of course there is an analogous result for the negative half-line.

(ii). Note that the two-sided condition \eqref{eq:decay} is owed to the fact that our proof relies on results from scattering theory,
where this condition appears naturally. It would be interesting to relax this condition to just boundedness. Moreover, it
would also be interesting to replace the constant solution by an arbitrary bounded solution. Unfortunately our method of proof
does not generalize to this situation.
\end{remark}

Moreover, there is also a similar result for the Kac--van Moerbeke lattice \cite{tjac} (again in Flaschka's variables)
\begin{equation}
\dot{\rho}(t,n)= \rho(n,t) \big( \rho(n+1,t)^2-\rho(n-1,t)^2\big).
\end{equation}

\begin{theorem}
Let $\rho_0>0$ and let $\rho(t)$ be a solution of the Kac--van Moerbeke lattice satisfying
\begin{equation}
\sum_{n\in\mathbb{Z}}|n| |\rho(n,t)-\rho_0|<\infty
\end{equation}
for one (and hence for all) $t\in\R$.
Suppose that for two different times $t_0<t_1$ and some constant $\delta>0$, we have
\begin{equation}
\sum_{n\geq M}\big(|\rho(n,t_j)-\rho_0|\big)\leq C\frac{1}{M^{(1+\delta)2M}},\quad M>0, \quad j\in\{0,1\}.
\end{equation}
Then
\begin{equation*}
\rho(n,t)=\rho_0, \quad (n,t)\in\Z\times\R.
\end{equation*}
\end{theorem}

In fact, in the next section we will prove this result for the entire Toda hierarchy which contains the Kac--van Moerbeke
as a special case (see Remark~\ref{rem:main}). Finally, we remark that our approach is inspired by the recent results from \cite{ART} and
\cite{JLMP} for the discrete Schr\"odinger equation.

\section{The main result}
\label{secth}

In this section we show that our main result extends to the entire Toda hierarchy (which will
cover the Kac--van Moerbeke hierarchy as well).
To this end, we introduce the Toda hierarchy using the standard Lax formalism
following \cite{bght} (see also \cite{ghmt}, \cite{tjac}).

Associated with two sequences $a^2(t)\neq 0, b(t)$ is a Jacobi operator
\begin{equation}\label{Hjacobi}
H(t)=a(t)S^++a^-(t)S^-+b(t)
\end{equation}
acting on sequences over $\mathbb{Z}$, where $S^{\pm}f(n)=f(n\pm1)$ are the usual shift operators. If we choose constants $c_0=1,c_j, 1\leq j\leq r$, $c_{r+1}=0$ and set
\begin{equation*}
P_{2r+2}(t)=\sum_{j=0}^rc_{r-j}\tilde{P}_{2j+2}(t),\qquad \tilde{P}_{2j+2}(t)=[H(t)^{j+1}]_+-[H(t)^{j+1}]_-,
\end{equation*}
where $[A]_{\pm}$ denotes the  upper and lower triangular parts of an operator with respect to the basis $\delta_m(n)=\delta_{m,n}$, here $\delta_{m,n}$ is the Kronecker delta.
Then the Toda hierarchy is given by the Lax equation
\begin{equation*}
\frac{d}{dt}H(t)-[P_{2r+2}(t),H(t)]=0,\qquad t\in\mathbb{R},
\end{equation*}
where $[A,B]=AB-BA$ is the usual commutator. Explicitly, setting
\begin{equation*}
\begin{split}
&g_j(n,t)=\sum_{l=0}^jc_{j-l}\tilde{g}_l(n,t),\quad \tilde{g}_l(n,t)=\langle\delta_n,H(t)^l\delta_n\rangle,\\
&h_j(n,t)=\sum_{l=0}^jc_{j-l}\tilde{h}_l(n,t),\quad \tilde{h}_l(n,t)=2a(n,t)\langle\delta_{n+1},H(t)^l\delta_n\rangle,
\end{split}
\end{equation*}
we obtain
\begin{equation}
\tl_r(a(t),b(t))=\begin{pmatrix}\dot{a}(t)-a(t)\big(g_{r+1}^+(t)-g_{r+1}(t)\big)\\ \dot{b}(t)-\big(h_{r+1}(t)-h_{r+1}^-(t)\big)\end{pmatrix}=0,\quad r\in\mathbb{N}_0.
\end{equation}
Here the dot denotes the derivative with respect to $t$ and $\mathbb{N}_0=\mathbb{N}\cup\{0\}$. Varying $r\in\N_0$ we obtain the Toda hierarchy and for $r=0$ we obtain
the Toda lattice \eqref{todeqfl}.
It is well known that the system $\text{TL}_r(a,b)=0$ with initial datum $(a_0,b_0)$ can be solved by using the inverse scattering transform. To this end suppose that
we have a solution of the Toda hierarchy, $\tl_r(a,b)=0$, satisfying
\begin{equation}
\sum_{n\in\mathbb{Z}}|n|\big(|a(n,t)-\tfrac{1}{2}|+|b(n,t)|\big)<\infty
\end{equation}
for one, and hence for all, $t\in\R$.

One introduces the scattering data
\begin{equation*}
S_\pm(H(t))=\{R_\pm(k,t),\text{ }|k|=1;\text{ }k_l,\gamma_{\pm,l}(t),\text{ }1\leq l\leq N\}
\end{equation*}
for the Jacobi operator $H(t)$. Here $R_\pm(k,t)$ are the left, right reflection coefficients, $\lambda_l= \frac{1}{2}(k_l+k_l^{-1})$ are the eigenvalues of $H(t)$, and
$\gamma_{\pm,l}(t)$ are the corresponding norming constants (see \cite[Chapter~11]{tjac} for precise definitions of these objects). Then the time evolution
of the scattering data is given by \cite[Theorem~13.8]{tjac}
\begin{equation}\label{eq:scatdat}
S_\pm(H(t))=\{R_\pm(k,0)\E^{\pm\alpha_r(k) t},\text{ }|k|=1;\text{ }k_l,\gamma_{\pm,l}(0)\E^{\pm\alpha_r(k_l)},\text{ }1\leq l\leq N\}
\end{equation}
where
\begin{equation*}
\alpha_r(k) = (k-k^{-1}) G_{0,r}\big(\frac{k+k^{-1}}{2}\big)
\end{equation*}
with $G_{0,r}(z)$ a monic polynomial of degree $r$ whose coefficients depend on the constants $c_j$ defining the Toda hierarchy (see \cite[Section~13.3]{tjac}).
The inverse scattering transform then amounts to computing the scattering data of the initial conditions $S_\pm(H(0))$ and then solving the inverse problem
to obtain the solution $(a(t),b(t))$ from $S_\pm(H(t))$ given via \eqref{eq:scatdat}.

Finally, we recall some properties of analytic functions, see also \cite{ART,JLMP}, which will play a crucial role in our proof.
We say that a function $f$ which is holomorphic outside a disc is of exponential type $\sigma_f$ if for $|z|$ big enough and some $\sigma>0$ we have
\begin{equation}
|f(z)|<\exp(\sigma|z|)
\end{equation}
In this case we define its indicator function by
\begin{equation}
h_f(\varphi)=\limsup_{r\to\infty}\frac{\log|f(r\E^{\I\varphi})|}{r}, \qquad \varphi\in[0,2\pi].
\end{equation}
It follows from the definition that if $f,g$ are two such functions, then
\begin{equation}\label{suma}
h_{f+g}\leq \max(h_f,h_g)
\end{equation}
and
\begin{equation}\label{product}
h_{fg}\leq h_f+f_g.
\end{equation}
Moreover, we also have the following identity
\begin{equation}\label{Corindicator}
h_f(\varphi)+h_f(\pi+\varphi)\geq 0.
\end{equation}
These facts are usually stated for entire functions. However, the key ingredient for the proof is the Phragm\'en--Lindel\"of theorem and thus one can easily adapt the proof of Theorem 1 from Chapter 8 in \cite{Levin} to show that it continuous to hold in the present situation. In particular, inequality \eqref{Corindicator} is still true in this case.

Now we are ready to establish our main result:

\begin{theorem}\label{Ith2}
Let $a_0>0$, $b_0\in\R$ be two given constants and
let $a(t),b(t)$ be a solution of the Toda hierarchy $\tl_r(a(t),b(t))=0$ satisfying
\begin{equation}
\sum_{n\in\mathbb{Z}}|n|\big(|a(n,t)-a_0|+|b(n,t)-b_0|\big)<\infty
\end{equation}
for one (and hence for all) $t\in\R$.
Suppose that for two different times $t_0<t_1$ and some constant $\delta>0$, we have
\begin{equation}\label{abdecay}
\sum_{n\geq M}\big(|a(n,t_j)-a_0|+|b(n,t_j)-b_0|\big)\leq C\frac{1}{M^{(1+\delta)2M}},\quad M>0, \quad j\in\{0,1\}.
\end{equation}
Then
\begin{equation*}
a(n,t)=a_0, \qquad b(n,t)=b_0, \quad (n,t)\in\Z\times\R.
\end{equation*}
\end{theorem}
\begin{proof}
Without loss of generality we choose $t_0=0$ and $t_1=1$. Moreover, by a simple transform $H\to \frac{1}{2a_0}(H-b_0)$
we can also assume $a_0=\tfrac{1}{2}$, $b_0=0$.
By \eqref{eq:scatdat} the time evolution of the reflection coefficient $R_+$ is given by
\begin{equation}
R_+(k,t)=R_+(k,0)\exp(\alpha_r(k)t),
\end{equation}
where $\alpha_r(k)=k^{r+1}-k^{-r-1}+\sum_{j=-r}^rd_jk^j$. Moreover, by \cite[eqs.~(2.15) and (2.18)]{ART} (note that $R_+(k)=\frac{\beta_+(k)}{\alpha(k)}$)
one has
\begin{equation}
\limsup_{|k|\to\infty}\frac{\log|R_+(k^{-1},t)|}{|k|}\leq 0,\qquad t\in\{0,1\}.
\end{equation}
Using \eqref{Corindicator}
\begin{equation*}
\begin{split}
&\limsup_{x\to\infty}\frac{\log|R_+( x^{-1},t)|}{x}\geq\\
& \limsup_{x\to\infty}\frac{\log|R_+( x^{-1},t)|}{x}+\limsup_{x\to\infty}\frac{\log|R_+( - x^{-1},t)|}{x} \geq 0,\qquad t\in\{0,1\}.
\end{split}
\end{equation*}
In particular, we conclude that $\limsup_{x\to\infty}\frac{\log|R_+( x^{-1},t)|}{x}=0$. A similar argument can be used to show $\limsup_{x\to\infty}\frac{\log|R_+( - x^{-1},t)|}{x}=0$.

On the other hand
\begin{equation*}
\begin{split}
&\limsup_{x\to\infty}\frac{\log|R_+( - x^{-1},1)|}{x}=\limsup_{x\to\infty}\frac{\log|R_+( - x^{-1},0)|}{x}+\\
&\limsup_{x\to\infty}\frac{\log|\exp(\alpha_r(- x^{-1})|}{x}=\lim_{x\to\infty}\frac{\log|\exp(x^{r+1}+O(x^{r}))|}{x} >0
\end{split}
\end{equation*}
Where in the last equality we have used that $r$ is even. If $r$ is odd apply the same argument to $R_+(k,0)=\exp(-\alpha_r)R_+(k,1)$.

In particular, this leads us to a contradiction and $R_+(k,0)=0$, whence $R_\pm(k,t)=0$. Consequently we have
a pure $N$ soliton solution (see \cite[eqn.~(14.109)]{tjac}) and since such a solution behaves like $a(n,t) -\frac{1}{2} \asymp k_1^{2|n|}$, where $\lambda_1=\frac{1}{2}(k_1+k_1^{-1})$ is the eigenvalue closest to
$[-1,1]$, we see that we get a contradiction to our assumption \eqref{abdecay} unless $N=0$.
\end{proof}

\begin{remark}\label{rem:main}
(i). By reflecting the coefficients $\tilde{a}(n,t)=a(-n-1,-t)$, $\tilde{b}(n,t)=b(-n,-t)$, which again satisfies $\tl_r(\ti{a}(t),\ti{b}(t))=0$,
we get a corresponding result on the negative half line.

(ii). Finally, since the Kac--van Moerbeke hierarchy can be obtained by setting $b=0$ in the odd equations of the
Toda hierarchy, $\km_r(a) = \tl_{2r+1}(a,0)$ (see \cite{mt}), this last result also covers the Kac--van Moerbeke hierarchy.
\end{remark}

\medskip

\noindent 
{\bf Acknowledgments.}
We are indebted to Yura Lyubarskii for discussions on this topic. 
I.\ A-R.\ gratefully acknowledges the hospitality of the Faculty of Mathematics, University of Vienna, Austria, during May, June 2016 where this research was initiated.

\end{document}